\documentclass[11pt]{article}
\usepackage[T1]{fontenc}
\usepackage{lmodern,amsmath,amsthm,amsfonts,amssymb,graphicx,float,microtype,thmtools,underscore,mathtools,anyfontsize,thm-restate,graphicx}
\usepackage[usenames,dvipsnames,svgnames,table]{xcolor}
\usepackage[unicode=true]{hyperref}
\hypersetup{
	colorlinks,
	linkcolor={black},
	citecolor={black},
	urlcolor={blue!60!black},
	pdftitle={2-Layer Graphs with Bounded Pathwidth},
	pdfauthor={David R. Wood}
}
\usepackage{breakurl}
\usepackage[noabbrev,capitalise,nameinlink]{cleveref}
\crefname{lem}{Lemma}{Lemmas}
\crefname{thm}{Theorem}{Theorems}
\crefname{cor}{Corollary}{Corollaries}
\crefname{prop}{Proposition}{Propositions}
\crefname{conj}{Conjecture}{Conjectures}
\crefname{open}{Open Problem}{Open Problems}
\crefformat{equation}{(#2#1#3)}
\Crefformat{equation}{Equation #2(#1)#3}
\usepackage[shortlabels]{enumitem}
\setlist[itemize]{topsep=0ex,itemsep=0ex,parsep=0.25ex}
\setlist[enumerate]{topsep=0ex,itemsep=0ex,parsep=0.25ex}
\newcommand{\defn}[1]{\textcolor{Maroon}{\emph{#1}}}

\newcommand{\NN}{\mathbb{N}}

\usepackage[longnamesfirst,numbers,sort&compress]{natbib}
\makeatletter
\def\NAT@spacechar{~}
\makeatother
\usepackage[bmargin=29mm,tmargin=29mm,lmargin=30mm,rmargin=30mm]{geometry}
\allowdisplaybreaks

\DeclarePairedDelimiter{\floor}{\lfloor}{\rfloor}

\renewcommand{\geq}{\geqslant}
\renewcommand{\leq}{\leqslant}
\renewcommand{\emptyset}{\varnothing}

\renewcommand{\thefootnote}{\fnsymbol{footnote}}
\allowdisplaybreaks
\theoremstyle{plain}
\newtheorem{thm}{Theorem}
\newtheorem{lem}[thm]{Lemma}

\newtheorem{cor}[thm]{Corollary}

\newtheorem{obs}[thm]{Observation}

\begin{document}
	
\title{\bf 2-Layer Graph Drawings with Bounded Pathwidth}
\author{David R. Wood\footnotemark[2]}
\date{12th July 2022}
	
\footnotetext[2]{\,School of Mathematics, Monash   University, Melbourne, Australia  (\texttt{david.wood@monash.edu}). Research supported by the Australian Research Council.}
	
\maketitle

\renewcommand{\thefootnote}{\arabic{footnote}}

\begin{abstract}
We determine which properties of 2-layer drawings characterise bipartite graphs of bounded pathwidth.
\end{abstract}
	
\section{Introducction}	

A \defn{2-layer drawing} of a bipartite graph $G$ with bipartition $\{A,B\}$ positions the vertices in $A$ at distinct points on a horizontal line, and positions the vertices in $B$ at distinct points on a different horizontal line, and draws each edge as a straight line-segment. 2-layer graph drawings are of fundamental importance in graph drawing research and have been widely studied~\citep{BDDEW19,Dujmovic-etal-Algo08,Suderman-IJCGA04,DFHKLMNRRSWW06,EadesWhitesides94,HS72,EadesWormald94,DujWhi-Algo04,DDEL14,Nagamochi05a,Nagamochi05b,DPW04,CSW04}. As illustrated in \cref{Caterpillar}, the following basic connection between 2-layer graph drawings and graph pathwidth\footnote{A \defn{path-decomposition} of a graph $G$ is a sequence $(B_1,\dots,B_n)$ of subsets of $V(G)$ (called \defn{bags}), such that $B_1\cup\dots\cup B_n=V(G)$, and for $1\leq i<j<k\leq n$ we have $B_i\cap B_k \subseteq B_j$; that is, for each vertex $v$ the bags containing $v$ form a non-empty sub-sequence of $(B_1,\dots,B_n)$. The \defn{width} of a path-decomposition $(B_1,\dots,B_n)$ is $\max_i|B_i|-1$. The \defn{pathwidth} of a graph $G$ is the minimum width of a path-decomposition of $G$. Pathwidth is a fundamental parameter in graph structure theory~\citep{RS-I,BRST91,Diestel-CPC95,Bodlaender-TCS98} with many connections to graph drawing~\citep{DMW05,Hliney-JCTB03,Suderman-IJCGA04,BDDEW19,DFHKLMNRRSWW06,BCDM20,DMY21,FLW-JGAA03} A \defn{caterpillar} is a tree such that deleting the leaves gives a path. It is a straightforward exercise to show that a connected graph has pathwidth 1 if and only if it is a caterpillar.} is folklore:

\begin{obs}
\label{Folklore}
A connected bipartite graph $G$ has a 2-layer drawing with no crossings if and only if $G$ is a caterpillar if and only if $G$ has pathwidth 1.
\end{obs}

\begin{figure}[!h]
	\centering
	\includegraphics{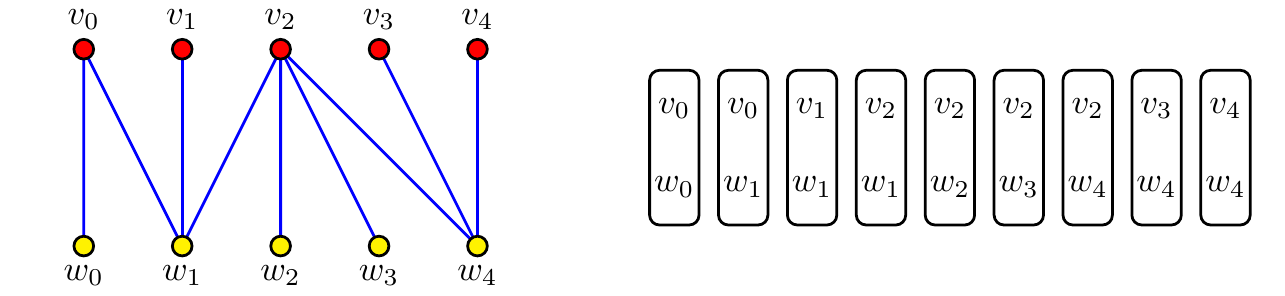}
	\caption{A caterpillar drawn on 2-layers with no crossings, and the corresponding path-decompostion with width 1. }
	\label{Caterpillar}
\end{figure}

Motivated by this connection, we consider (and answer) the following question: what properties of 2-layer drawings characterise bipartite graphs of bounded pathwidth?

A \defn{matching} in a graph $G$ is a set of edges in $G$, no two of which are incident to a common vertex. A \defn{$k$-matching} is a matching of size $k$. In a 2-layer drawing of a graph $G$, a \defn{$k$-crossing} is a set of $k$ pairwise crossing edges (which necessarily is a $k$-matching). 
Excluding a $k$-crossing is not enough to guarantee bounded pathwidth. For example, as illustrated in \cref{TwoLayerTree}, if $T_h$ is the complete binary tree of height $h$, then $T_h$ has a 2-layer drawing with no 3-crossing, but it is well known that $T_h$ has pathwidth $\floor{h/2}+1$. Even stronger, if $G_h$ is the $h\times h$ square grid graph, then $G_h$ has a 2-layer drawing with no 3-crossing, but $G_h$ has treewidth and pathwidth $h$.

\begin{figure}[!h]
\centering
\includegraphics{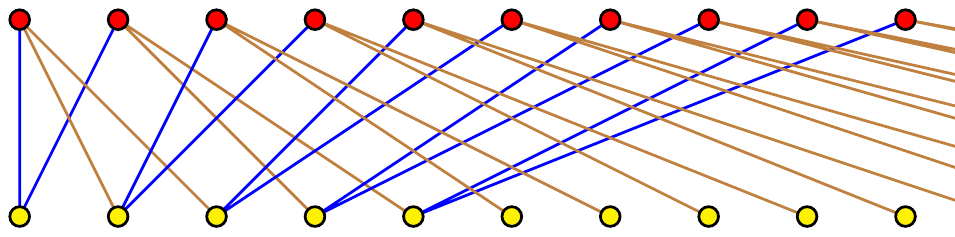}
\includegraphics{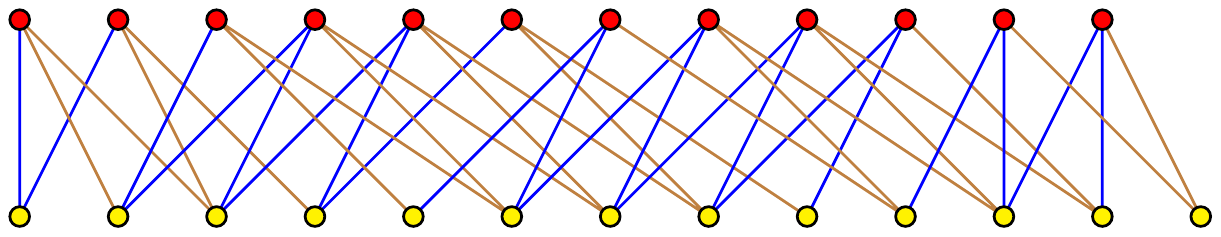}
\caption{2-layer drawings of a complete binary tree and a $5\times 5$ grid. There is no 3-crossing since each edge is assigned one of two colours, so that monochromatic edges do not cross. }
\label{TwoLayerTree}
\end{figure}



\citet{ALFS20} showed that every graph that has a 2-layer drawing with at most $k$ crossings on each edge has pathwidth at most $k+1$. However, this property does not characterise bipartite graphs with bounded pathwidth. For example, as illustrated in \cref{TwoLayerStar}, if $S_n$ is the 1-subdivision of the $n$-leaf star, then $S_n$ is bipartite with pathwidth 2, but in every 2-layer drawing of $S_n$, some edge has at least $(n-1)/2$ crossings. 

\begin{figure}[!h]
	\centering
	\includegraphics{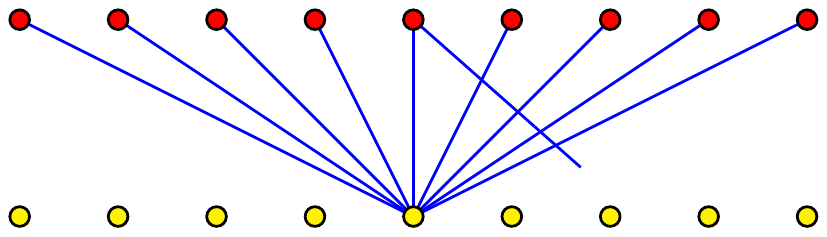}
	\caption{Every 2-layer drawing of $S_9$ has at least 4 crossings on some edge. }
	\label{TwoLayerStar}
\end{figure}


These examples motivate the following definition. A set $S$ of edges in a 2-layer drawing is \defn{non-crossing} if no two edges in $S$ cross. In a 2-layer drawing of a graph $G$, an \defn{$(s,t)$-crossing} is a pair $(S,T)$ where  $S$ is a non-crossing $s$-matching, $T$ is a non-crossing $t$-matching, and every edge in $S$ crosses every edge in $T$; as illustrated in \cref{stCrossing}. 

\begin{figure}[!h]
\centering
\includegraphics{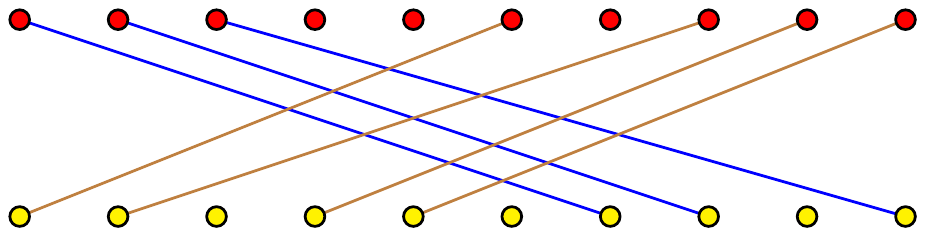}
\caption{Example of a $(3,4)$-crossing.}
\label{stCrossing}
\end{figure}

We show that excluding a $k$-crossing and an $(s,t)$-crossing guarantees bounded pathwidth. 

\begin{thm}
\label{2LayerPathwidth}
For all $k,s,t\in\NN$, every bipartite graph $G$ that has a 2-layer drawing with no $(k+1)$-crossing and no $(s,t)$-crossing has pathwidth at most 
$8k^2(t-1) +4k^2 (s-1)^2(s-2)+5k+4$.
\end{thm}

We prove the following converse to \cref{2LayerPathwidth}. 

\begin{thm}
\label{Converse}
For any $k\in\mathbb{N}$ every bipartite graph $G$ with pathwidth at most $k$ has a 2-layer drawing with no $(k+2)$-crossing and no $(k+1,k+1)$-crossing. 
\end{thm}

\cref{2LayerPathwidth,Converse} together establish the following rough characterisation of  bipartite graphs with bounded pathwidth, thus answering the opening question.

\begin{cor}
A class $\mathcal{G}$ of bipartite graphs has bounded pathwidth if and only if 
there exists $k,s,t\in\NN$ such that every graph in $\mathcal{G}$ has a 2-layer drawing with no $k$-crossing and no $(s,t)$-crossing.
\end{cor}


%




\section{Proofs}

\begin{lem}
\label{Useful}	
Let $G$ be a bipartite graph with bipartition $A,B$, where each vertex in $A$ has degree at least $1$ and each vertex in $B$ has degree at most $d$. Assume that $G$ has a 2-layer drawing with no $(k+1)$-crossing and no non-crossing $(\ell+1)$-matching. Then $|A|\leq k\ell d$.
\end{lem}

\begin{proof}
Let $X$ be a set of edges in $G$ with exactly one edge in $X$ incident to each vertex in $A$. So $|X|=|A|$. Let $E_1,\dots,E_d$ be the partition of $X$, where for each edge $vw\in E_i$, if $v\in A$ and $w\in B$, then $v$ is the $i$-th neighbour of $w$ with respect to $\preceq_A$. So each $E_i$ is a matching. By Dilworth's Theorem, there is a partition $E_{i,1},\dots,E_{i,k}$ of $E_i$ such that edges in each $E_{i,j}$ are pairwise non-crossing. By assumption, $|E_{i,j}|\leq \ell$. Thus $|A|=|X|\leq k\ell d$. 
\end{proof}

\begin{proof}[Proof of \cref{2LayerPathwidth}] (We make no effort to optimise the bound on the pathwidth of $G$.) Let $A$ and $B$ be the layers in the given drawing of $G$. Let $\preceq_A$ be the total order of $A$, where $v\prec_A w$ if $v$ is to the left of $w$ in the drawing. Define $\preceq_B$ similarly. Let $\preceq$ be the poset on $E(G)$, where $vw\preceq xy$ if $v\preceq_A x$ and $w\preceq_B y$. Two edges of $G$ are comparable under $\preceq$ if and only if they do not cross. Thus every antichain under $\preceq$ is a matching of pairwise crossing edges. Hence there is no antichain of size $k+1$ under $\preceq$. By Dilworth's Theorem~\citep{Dilworth50}, there is a partition of $E(G)$ into $k$ chains under $\preceq$. Each chain is a caterpillar forest, which can be oriented with outdegree at most 1 at each vertex. So each vertex has out-degree at most $k$. For each vertex $v$, let $N^+_G[v]:=\{w\in V(G): \overrightarrow{vw}\in E(G)\}\cup\{v\}$, which has size at most $k+1$. 

Let $X=\{e_1,\dots,e_n\}$ be a maximal non-crossing matching, where $e_1\prec e_2\prec\dots\prec e_n$. Let $Y_0$ be the set of vertices of $G$ strictly to the left of $e_1$. For $i\in\{1,2,\dots,n-1\}$, let $Y_i$ be the set of vertices of $G$ strictly between $e_i$ and $e_{i+1}$. Let $Y_n$ be the set of vertices of $G$ strictly to the right of $e_n$. By the maximality of $X$, each set $Y_i$ is independent. For $i\in\{0,1,\dots,n\}$, arbitrarily enumerate $Y_i=\{v_{i,1},\dots,v_{i,m_i}\}$. Note that $v_{i,j}$ is an end-vertex of no edge in $X$  (for all $i,j$).

For each $i\in\{1,\dots,n\}$, if $e_i=xy$ then let $N_i=N^+_G[x]\cup N^+_G[y]$ and let $V_i$ be the set consisting of $N_i$ along with every vertex $v\in V(G)$ such that some arc $\overrightarrow{xv} \in E(G)$ crosses $e_i$. Note that $|N_i|\leq 2(k+1)$. For each $i\in\{0,1,\dots,n\}$ and $j\in\{1,\dots,m_i\}$, 
let $V_{i,j}:=(V_i\cup V_{i+1})\cup N^+_G[v_{i,j}]$ where $V_0:=V_{n+1}:=\emptyset$.

We now prove that
\begin{align}
\label{PathDec}
(V_{0,1},\dots,V_{0,m_0};V_1;V_{1,1},\dots,V_{1,m_0};\dots;V_n;V_{n,1},\dots,V_{n,m_n})
\end{align}
is a path-decomposition of $G$. We first show that each vertex $v$ is in some bag. If $v$ is an end-vertex of some edge $e_i$, then $v\in V_i$. Otherwise $v=v_{i,j}$ for some $i,j$, implying that $v\in V_{i,j}$, as desired. We now show that each vertex $v$ is in a consecutive sequence of bags. Suppose that $v\in V_i\cap V_p$ and $i<j<p$. Thus $e_i\prec e_j\prec e_p$. Our goal is to show that $v\in V_j$. If $v$ is an end-vertex of $e_j$, then $v\in V_j$. So we may assume that $v$ is not an end-vertex of $e_j$. By symmetry, we may assume that $v$ is to the left of the end-vertex of $e_j$ in the same layer as $v$. Thus, $v$ is not an end-vertex of $e_p$. Since $v\in V_p$, there is an arc $\overrightarrow{yv}$ that crosses $e_p$ or $y$ is an end-vertex of $e_p$. Since $e_j\prec e_p$, this arc $\overrightarrow{yv}$ crosses $e_j$. Thus $v\in V_j$, as desired. This shows that $v$ is in a consecutive (possibly empty) sequence of bags $V_i,V_{i+1},\dots,V_j$. 
If $v\in V_i$ then $v\in V_{i,j}$ for all $j\in\{1,\dots,m_i\}$, and 
$v\in V_{i-1,j}$ for all $j\in\{1,\dots,m_{i-1}\}$. It remains to consider the case in which $v$ is in no set $V_i$. Since the end-vertices of $e_i$ are in $V_i$, we have that $v=v_{i,j}$ for some $i,j$. Since $Y_i$ is an independent set, $v$ is adjacent to no other vertex in $Y_i$. Moreover, if there is an arc $\overrightarrow{xv}$ in $G$, then either $x$ is an end-vertex of $e_i$ or $e_{i-1}$, or $\overrightarrow{xv}$ crosses $e_{i-1}$ or $e_i$, implying $v$ is in $V_{i-1}\cup V_i$, which is not the case. Hence $v$ has indegree 0, implying $V_{i,j}$ is the only bag containing $v$. This completes the proof that $v$ is in a consecutive set of bags in \eqref{PathDec}. Finally, we show that the end-vertices of each edge are in some bag. Consider an arc $\overrightarrow{vw}$ in $G$. If $v=v_{i,j}$ for some $i,j$, then $v,w\in V_{i,j}$, as desired. Otherwise, $v$ is an end-vertex of some $e_i$, implying $v,w\in V_{i}$, as desired. Hence the sequence in \eqref{PathDec} defines a path-decomposition of $G$. We now bound the width of this path-decomposition. 

For $i,j\in\{0,1,\dots,n\}$, let $Y_{i,j}$ be the set of vertices $v\in Y_i$ such that there is an arc $\overrightarrow{xv}$ in $G$ with $x\in Y_j$. Suppose that $|Y_{i,j}| \geq 2k^2|j-i|+1$ for some $i,j\in\{0,1,\dots,n\}$. Without loss of generality, $i\leq j$ and 
there exists $Z\subseteq Y_{i,j}\cap A$ with $|Z| \geq k^2(j-i)+1$. Let $H$ be the subgraph of $G$ consisting of all arcs $\overrightarrow{xv}$ in $G$ with $x\in Y_j\cap B$ and $v\in Z$ (and their end-vertices). If $H$ has a non-crossing $(j-i+1)$-matching $M$, then $(X\setminus\{e_{i+1},\dots,e_j\})\cup M$ is a non-crossing matching in $G$ larger than $X$, thus contradicting the choice of $X$. Hence $H$ has no non-crossing $(j-i+1)$-matching. By construction, $H$ has no $(k+1)$-crossing, every vertex in $V(H)\cap A$ has degree at least 1 in $H$, and every vertex in $V(H)\cap B$ has degree at most $k$ in $H$. By \cref{Useful} applied to $H$ with $\ell=j-i$ and $d=k$, we have $|Z|=|V(H)\cap A|\leq k^2(j-i)$, which is a contradiction. Hence $|Y_{i,j}| \leq  2k^2|j-i|$ for all $i,j\in\{0,1,\dots,n\}$. 

For $i\in\{1,\dots,n\}$, let $P_i$ be the set of vertices $v$ in $G$ for which there is an arc $\overrightarrow{xv}$ in $G$ that crosses $e_{i-s+1},e_{i-s+2},\dots,e_i$ or crosses $e_{i},e_{i+1},\dots,e_{i+s-1}$. Suppose that $P_i\geq 4k^2(t-1)+1$. Without loss of generality, there exists $Q\subseteq P_i \cap A$ with $|Q|\geq k^2(t-1)+1$ such that for each vertex $v\in Q$ there is an arc $\overrightarrow{xv}$ in $G$ that crosses $e_{i-s+1},e_{i-s+2},\dots,e_i$. Let $H$ be the subgraph of $G$ consisting of all such arcs and their end-vertices. So $V(H)\cap A=Q$. If $H$ has a non-crossing $t$-matching $M$, then $(\{e_{i-s+1},e_{i-s+2},\dots,e_i\},M)$ is an $(s,t)$-crossing. Thus $H$ 
has no non-crossing $t$-matching. By construction, $H$ has no $(k+1)$-crossing, every vertex in $V(H)\cap A$ has degree at least 1 in $H$, and every vertex in $V(H)\cap B$ has degree at most $k$ in $H$. By \cref{Useful} applied to $H$ with $\ell=t-1$ and $d=k$, we have $|Q| = |V(H)\cap A|\leq k^2(t-1)$, which is a contradiction. Hence $|P_i| \leq 4k^2(t-1)$ for all $i\in\{1,\dots,n\}$. 

Consider a bag $V_i$, which consists of the end-vertices of $e_i$, along with every vertex $v\in V(G)$ such that some arc $\overrightarrow{xv} \in E(G)$ crosses $e_i$. 
Thus 
\begin{align*}
|V_i| & = |N_i|+  |P_i| + \sum_{a,b\in\{0,1,\dots,s-2\}} \!\!\!\!\!\!\!\! |Y_{i-a,i+b}|\\
& \leq 2(k+1)+4k^2(t-1) +  \sum_{a,b\in\{0,1,\dots,s-2\}} \!\!\!\!\!\!\!\! 2k^2|(i+b)-(i-a)|\\
& =  2(k+1) +4k^2(t-1) +2k^2 
\sum_{a,b\in\{0,1,\dots,s-2\}} \!\!\!\!\!\!\!\! (a+b)\\
& = 2(k+1) +4k^2(t-1) + 2k^2 \left( 
(s-1) \left( \sum_{a\in\{0,1,\dots,s-2\}} \!\!\! a \right) 
+
(s-1) \left( \sum_{b\in\{0,1,\dots,s-2\}} \!\!\! b \right) \right)\\
& =  2(k+1) +4k^2(t-1) +2k^2 (s-1)^2(s-2) .
\end{align*}
Hence 
\begin{align*}
|V_{i,j}|  \leq |V_i|+|V_{i+1}|+(k+1) 
& \leq 4(k+1) +8k^2(t-1) +4k^2 (s-1)^2(s-2)+(k+1)\\
& \leq 8k^2(t-1) +4k^2 (s-1)^2(s-2)+5(k+1).
\end{align*}
Therefore the path-decomposition of $G$ defined in \eqref{PathDec} has width at most 
$8k^2(t-1) +4k^2 (s-1)^2(s-2)+5k+4$. 
\end{proof}


\begin{proof}[Proof of \cref{Converse}]
Let $(X_1,\dots,X_n)$ be a path-decomposition of $G$ with width $k$. 
Let $\ell(v):=\min\{i:v\in X_i\}$ and $r(v):=\max\{i:v\in X_i\}$ for each $v\in V(G)$. 
We may assume that $\ell(v)\neq\ell(w)$ for all distinct $v,w\in V(G)$. 
Let $\{A,B\}$ be a bipartition of $G$. 
Consider the 2-layer drawing of $G$, in which each $v\in A$ is at $(\ell(v),0)$, each $v\in B$ is at $(\ell(v),1)$, and each edge is straight. 


Suppose that $\{v_1w_1,\dots,v_{k+2}w_{k+2}\}$ is a $(k+2)$-crossing in this drawing, where $v_i\in A$ and $w_i\in B$. Without loss of generality, 
\begin{align}
\label{Crossing}
\ell(v_1) < \ell(v_2)  < \dots < \ell(v_{k+2}) \quad\text{and} \quad 
\ell(w_{k+2}) < \ell(w_{k+1}) < \dots < \ell(w_1).
\end{align}
For each $i\in\{1,\dots,k+2\}$, 
if $\ell(v_i)<\ell(w_i)$ then let $I_i:=\{\ell(v_i),\dots,\ell(w_i)\}$; 
otherwise let $I_i:=\{\ell(w_i),\dots,\ell(v_i)\}$.
By \eqref{Crossing}, $I_i\cap I_j\neq\emptyset$ for distinct $i,j\in\{1,\dots,k+2\}$.
By the Helly property for intervals, there exists $p\in I_1\cap\dots\cap I_{k+2}$.
Thus $v_i$ or $w_i$ is in $X_p$ for each $i\in\{1,\dots,k+2\}$. 
Hence $|X_p|\geq k+2$, which is a contradiction. 
Therefore there is  no $(k+2)$-crossing.
 
Suppose that $(\{v_1w_1,\dots,v_sw_s\},\{x_1y_1,\dots,x_sy_s\})$ is a $(s,s)$-crossing in this drawing, where $v_i,x_i\in A$ and $w_i,y_i\in B$. 
Without loss of generality, 
\begin{align*}
& \ell(v_1) < \dots < \ell(v_s) < \ell(x_1) < \dots < \ell(x_s)\quad\text{and}\\
& \ell(y_1) < \dots < \ell(y_s) < \ell(w_1) < \dots < \ell(w_s).
\end{align*}
We claim that $k\geq s$. 
If $\ell(v_s) < \ell(w_1)$ then 
$\ell(v_1) < \dots < \ell(v_s) < \ell(w_1) < \dots < \ell(w_s)$, implying 
$v_1,\dots,v_s,w_1\in X_{\ell(w_1)}$, and $k\geq s$. 
If $\ell(y_s) < \ell(x_1)$ then 
$\ell(y_1) < \dots < \ell(y_s) < \ell(x_1) <  \dots < \ell(x_s)$, 
implying 
$y_1,\dots,y_s,x_1\in X_{\ell(x_1)}$, and $k\geq s$. 
Now assume that 
$\ell(w_1) < \ell(v_s) $ and $\ell(x_1) < \ell(y_s)$.
Thus
$\ell(w_1) < \ell(v_s) < \ell(x_1) < \ell(y_s)$, which 
is a contradiction since 
$\ell(y_s) < \ell(w_1)$.
Hence $s\leq k$ and the drawing of $G$ has no  $(k+1,k+1)$-crossing. 
\end{proof}


{
\let\oldthebibliography=\thebibliography
\let\endoldthebibliography=\endthebibliography
\renewenvironment{thebibliography}[1]{%
\begin{oldthebibliography}{#1}%
	\setlength{\parskip}{0ex}%
	\setlength{\itemsep}{0ex}%
}{\end{oldthebibliography}}

\def\soft#1{\leavevmode\setbox0=\hbox{h}\dimen7=\ht0\advance \dimen7
	by-1ex\relax\if t#1\relax\rlap{\raise.6\dimen7
		\hbox{\kern.3ex\char'47}}#1\relax\else\if T#1\relax
	\rlap{\raise.5\dimen7\hbox{\kern1.3ex\char'47}}#1\relax \else\if
	d#1\relax\rlap{\raise.5\dimen7\hbox{\kern.9ex \char'47}}#1\relax\else\if
	D#1\relax\rlap{\raise.5\dimen7 \hbox{\kern1.4ex\char'47}}#1\relax\else\if
	l#1\relax \rlap{\raise.5\dimen7\hbox{\kern.4ex\char'47}}#1\relax \else\if
	L#1\relax\rlap{\raise.5\dimen7\hbox{\kern.7ex
			\char'47}}#1\relax\else\message{accent \string\soft \space #1 not
		defined!}#1\relax\fi\fi\fi\fi\fi\fi}

\end{document}